\documentclass{jfl}

\usepackage{amsmath}
\usepackage{amssymb}
\usepackage{amsthm}
\usepackage{xy}
\xyoption{poly}
\xyoption{frame}

\usepackage{times}

 \newenvironment{claimproof}{\begin{proof}}{\renewcommand{\qedsymbol}{$\dashv$}\end{proof}}

\newcommand{\B}{\mathcal{P}_{fin}(\mathbb{N})}
\author{Henry Towsner}
\title{A Simple Proof and Some Difficult Examples for Hindman's Theorem}
\date{\today}

\begin{abstract}
  We give a short, explicit proof of Hindman's Theorem that in every finite coloring of the integers, there is an infinite set all of whose finite sums have the same color.  We give several examples of colorings of the integers which do not have computable witnesses to Hindman's Theorem.
\end{abstract}

\begin{document}
\maketitle

\section{Introduction}
Hindman's Theorem is:
\begin{theorem}
  If $c:\mathbb{N}\rightarrow[1,r]$ is given then there are an $i\in[1,r]$ and an an infinite set $S$ such that $c(s)=i$ whenever $s$ is the sum of one or more distinct elements of $S$.
\end{theorem}

There are three standard proofs of Hindman's theorem: the original combinatorial argument (\citend{hindman74}), a streamlined combinatorial argument (\citend{baumgartner74}), and the Galvin-Glazer proof using ultrafilters (see \citend{comfort77} or \citend{hindman98}).  The original proof is generally considered quite difficult (see, for instance, the comments on it in \citend{Hindman2005}), but work in reverse mathematics shows that it is also, at least in the sense of reverse mathematics, the simplest of the three proofs.  Specifically, Blass, Hirst, and Simpson have shown (\citend{blass87}) that Hindman's proof can be formalized in the system $\mathbf{ACA_0^+}$, while Baumgartner's proof can be formalized in the stronger system $\mathbf{\Pi^1_2-TI_0}$.  The Galvin-Glazer proof was analyzed in (\citend{towsner09:hindman_unwind}), where an even stronger system was used to formalize it.  
  (The definitions and significance of all these systems of reverse mathematics may be found in (\citend{simpson99}).)

The work in \citend{towsner09:hindman_unwind} demonstrated a striking analogy between the structures of Baumgartner and Galvin-Glazer proofs: roughly speaking, both proofs prove an intermediate theorem that a structure weaker than that promised by Hindman's Theorem exists, then repeat the same argument with one step replaced by the intermediate theorem.  Hindman's proof does not have this structure, but comparison of the proofs suggests that the corresponding intermediate would be the structure given by Theorem \ref{halfmatch} below.  With the use of this intermediate, we can give a new proof similar to Hindman's which is provable in the slightly stronger system $\mathbf{ACA^+}$.

\citend{blass87} also gives a lower bound for the reverse mathematical strength of Hindman's Theorem by constructing a computable coloring such that $\mathbf{0}'$ is computable in any set witnessing Hindman's Theorem, one one such that no set witnessing Hindman's Theorem is $\Delta_2$.  In particular, Hindman's Theorem implies $\mathbf{ACA_0}$ over $\mathbf{RCA_0}$. We describe a flexible method for giving colorings for which Hindman's Theorem is difficult to solve, including examples which show that certain aspects of our proof are optimal.

We are grateful to Mathias Beiglb\"ock and Carl Mummert for many helpful discussions about the many facets of Hindman's Theorem.

\section{A Simple Proof of Hindman's Theorem}
It is standard (see \citend{baumgartner74}) to take advantage of the fact that Hindman's Theorem is equivalent to a similar statement about unions of finite sets.  We will freely equate $\mathcal{P}_{fin}(\mathbb{N})$ with $\mathbb{N}$, using the fact that there is a computable bijection between the two sets.
\begin{definition}
  If $\mathcal{S}\subseteq\mathcal{P}_{fin}(\mathbb{N})$, we write $NU(\mathcal{S})$ for the set of \emph{non-empty unions} from $\mathcal{S}$, those non-empty $T$ which are the union of finitely many elements of $\mathcal{S}$.

  We say $\mathcal{S}\subseteq\mathcal{P}_{fin}(\mathbb{N})$ is IP if it is closed under finite unions and contains an infinite set of pairwise disjoint elements.

  If $B\in\mathcal{S}$, we will write
\[\mathcal{S}-B:=\{T\in\mathcal{S}\mid T\cap B=\emptyset\},\]
and if $\mathcal{B}\subseteq\mathcal{S}$ then
\[\mathcal{S}-\mathcal{B}:=\mathcal{S}-\bigcup\mathcal{B}.\]
\end{definition}
Then subtraction is a strong form of set difference, where we remove not only $B$, but also anything that intersects $B$.

The following theorem is easily seen to imply Hindman's Theorem (consider the map taking a number $n$ to the set of places which are $1$ in the binary expansion of $n$).  (With more work, it can be seen to follow from Hindman's Theorem as well.)
\begin{theorem}[Finite Unions Theorem]
  If $c:\mathcal{P}_{fin}(\mathbb{N})\rightarrow[1,r]$ is given then there are an $i\in[1,r]$ and an IP set $\mathcal{S}$ such that $c(S)=i$ for every $S\in\mathcal{S}$.
\end{theorem}

We introduce two weak notions which will characterize our intermediate steps:
\begin{definition}
  We say $\mathcal{D}$ \emph{half-matches} $B$ if there is a $D\in \mathcal{D}$ such that $c(B)=c(D\cup B)$.  We say $\mathcal{D}$ half-matches a set $\mathcal{B}$ if $\mathcal{D}$ half-matches every $B\in\mathcal{B}$.

  We say $\mathcal{D}$ \emph{full-matches} $B$ if there is a $D\in \mathcal{D}$ such that $c(D)=c(B)=c(D\cup B)$.  We say $\mathcal{D}$ full-matches a set $\mathcal{B}$ if $\mathcal{D}$ full-matches every $B\in\mathcal{B}$.
\end{definition}

\begin{lemma}[$\mathbf{RCA_0}$]
  Let $\mathcal{S}$ be an IP set, let $\mathcal{B}\subseteq \mathcal{S}$ be finite, and let $c:NU(\mathcal{S})\rightarrow [1,r]$ be given.  Then either:
  \begin{itemize}
    \item There is a finite $\mathcal{D}\subseteq \mathcal{S}-\mathcal{B}$ such that for every $S\in \mathcal{S}-\mathcal{B}-\mathcal{D}$, there is a $D\in NU(\mathcal{D})$ such $\mathcal{B}$ does not half-match $D\cup S$, or
    \item There is an IP set $\mathcal{T}\subseteq\mathcal{S}-\mathcal{B}$ such that $\mathcal{B}$ half-matches $\mathcal{T}$.
  \end{itemize}
\label{badsequencelemma}
\end{lemma}
\begin{proof}
  Suppose the first condition fails; that is, for any finite $\mathcal{D}\subseteq\mathcal{S}-\mathcal{B}$, there is an $S\in\mathcal{S}-\mathcal{B}-\mathcal{D}$ such that $\mathcal{B}$ half-matches $D\cup S$ for every $D\in NU(\mathcal{D})$.

  We inductively construct a sequence $\mathcal{D}_0\subseteq\mathcal{D}_1\subseteq\cdots$ of finite subsets of $\mathcal{S}-\mathcal{B}$ such that whenever $D\in NU(\mathcal{D}_n)\setminus\mathcal{D}_n$, $\mathcal{B}$ half-matches $D$.  Set $\mathcal{D}_0:=\{D_0\}$ for an arbitrary $D_0\in\mathcal{S}-\mathcal{B}$.  Given $\mathcal{D}_n$, since the first condition fails and $NU(\mathcal{D}_n)$ is finite, there is an $S\in\mathcal{S}-\mathcal{B}-\mathcal{D}_n$ such that for every $D\in NU(\mathcal{D}_n)$, $\mathcal{B}$ half-matches $D\cup S$.  Let $\mathcal{D}_{n+1}:=\mathcal{D}_n\cup\{S\}$.  Then for any $D\in NU(\mathcal{D}_{n+1})\setminus\mathcal{D}_{n+1}$, either $D\in NU(\mathcal{D}_n)\setminus\mathcal{D}_n$, in which case $\mathcal{B}$ half-matches $D$ by IH, or $D=D'\cup S$ for some $D'\in\mathcal{D}_n$, in which case $\mathcal{B}$ half-matches $D$ by choice of $S$.

  Let $\mathcal{D}:=\bigcup_n\mathcal{D}_n=\{D_0,D_1,\ldots\}$.  Let $\mathcal{D}':=\{D_{2i}\cup D_{2i+1}\mid i\in\mathbb{N}\}$.  Then if $D\in NU(\mathcal{D}')$, $D\in NU(\mathcal{D}_n)\setminus\mathcal{D}_n$ for some $n$, so $\mathcal{B}$ half-matches $D$.
\end{proof}

\begin{lemma}[$\mathbf{RCA}$]
  If $\mathcal{S}$ is an IP set and $c:NU(\mathcal{S})\rightarrow [1,r]$ then there is a finite collection $\mathcal{B}\subseteq \mathcal{S}$ and an IP set $\mathcal{T}\subseteq \mathcal{S}-\mathcal{B}$ such that $\mathcal{B}$ half-matches $\mathcal{T}$.
\label{halfmatch}
\end{lemma}
\begin{proof}
  Pick an arbitrary element $Q\in\mathcal{S}$, and set $\mathcal{B}_1:=\{Q\}$ and $\mathcal{S}'_1:=\mathcal{S}-\mathcal{B}_1$.  Given $\mathcal{B}_i,\mathcal{S}'_i$, apply Lemma \ref{badsequencelemma}.  If the second condition holds, we are finished.  Otherwise let $\mathcal{D}_{i+1}$ be given by the first part, let $\mathcal{B}_{i+1}:=NU(\mathcal{B}_i\cup\mathcal{D}_{i+1})$, and let $\mathcal{S}'_{i+1}:=\mathcal{S}'_i-\mathcal{B}_{i+1}$.  

Suppose that we reach $\mathcal{B}_r,\mathcal{S}'_r$ without terminating.  Then for any $S\in \mathcal{S}'_r$, we may choose a sequence $D_r,\ldots,D_2$ with $D_i\in NU(\mathcal{D}_i)$ and for each $i$, $\mathcal{B}_{i-1}$ fails to half-match $S\cup\bigcup_{j=i}^r D_j$.  Let $D_1:=Q$.  Then for each $i< i'$, since $\bigcup_{j=i}^{i'-1}D_j\in \mathcal{B}_{i'-1}$, $c(S\cup\bigcup_{j=i}^r D_j)\neq c(S\cup \bigcup_{j=i'}^r D_j)$.  But since there are $r$ colors, there must be some $i$ such that $c(S)=c(S\cup\bigcup_{j=i}^rD_j)$.  Therefore we may take $\mathcal{B}:=\mathcal{B}_r$ and $\mathcal{T}:=\mathcal{S}'_r$.
\end{proof}


\begin{lemma}[$\mathbf{ACA}$]
    Let $\mathcal{S}$ be an IP set and let $c:\mathcal{S}\rightarrow [1,r]$ be given.  Then either:
  \begin{itemize}
  \item There is an IP $\mathcal{S}'\subseteq \mathcal{S}$ and some $i\in[1,r]$ such that $c(S)\neq i$ for every $S\in \mathcal{S}'$, or
  \item There is a finite collection $\mathcal{B}\subseteq \mathcal{S}$ and an IP set $\mathcal{T}\subseteq \mathcal{S}-\mathcal{B}$ such that $\mathcal{B}$ full-matches $\mathcal{T}$.
  \end{itemize}
\label{fullmatchmaybe}
\end{lemma}
\begin{proof}
  Construct sequences $\mathcal{B}_2,\ldots,\mathcal{B}_n,\ldots$, $\mathcal{T}_1,\ldots,\mathcal{T}_n,\ldots$, and colorings $c_1,\ldots,c_n,\ldots$ as follows: let $c_1:=c$ and $\mathcal{T}_1:=\mathcal{S}$.  Given $c_i,\mathcal{T}_i$, let $\mathcal{B}_{i+1},\mathcal{T}_{i+1}$ be the witness given by Lemma \ref{halfmatch}.  Define $c_{i+1}$ on $\mathcal{T}_{i+1}$ by setting $c_{i+1}(S):=\langle B,c_i(S)\rangle$ where $B\in\mathcal{B}_{i+1}$ is such that $c_i(S)=c_i(S\cup B)$.

If there is some $n$ such that for every $S\in\mathcal{T}_n$ there is a $B\in NU(\bigcup_{i\leq n}\mathcal{B}_i)$ such that $c(S)=c(B)=c(S\cup B)$ then $\mathcal{T}_n$ and $\bigcup_{i\leq n}\mathcal{B}_i$ witness the second possibility.

Otherwise, for each $n$ we may choose a $T_n\in\mathcal{T}_n$ such that there is no $B\in NU(\bigcup_{i\leq n}\mathcal{B}_i)$ such that $c(T_n)=c(B)=c(T_n\cup B)$.  By the pigeonhole principle, we may choose an infinite subsequence $\{T_{i_n}\}$ such that $c$ is constantly some fixed $q\in[1,r]$ on $\{T_{i_n}\}$ (but not necessarily on $NU(\{T_{i_n}\})$).  For each $T_{i_n}$, we may choose a sequence $B_1\in\mathcal{B}_1,\ldots,B_{i_n}\in\mathcal{B}_{i_n}$ such that $c(T_{i_n})=c(T_{i_n}\cup B)$ for every $B\in NU(\{B_i\})$.  In particular, it must be that $c(B)\neq q$.

Then by K\"onig's Lemma, we may choose an infinite sequence $\{B_i\}$ such that $c(B)\neq q$ for any $B\in NU(\{B_1,\ldots,B_n,\ldots\})$.
\end{proof}
Note that, when the second clause holds in the preceeding lemma, the set $\mathcal{T}$ is computable from $c$ and $\mathcal{S}$.

\begin{lemma}[$\mathbf{ACA}$]
    Let $\mathcal{S}$ be an IP set and let $c:\mathcal{S}\rightarrow [1,r]$ be given.  Then either:
  \begin{itemize}
  \item There is an IP $\mathcal{S}'\subseteq \mathcal{S}$ such that $c$ is constant on $\mathcal{S}'$, or
  \item There is a finite collection $\mathcal{B}\subseteq \mathcal{S}$ and an IP set $\mathcal{T}\subseteq \mathcal{S}-\mathcal{B}$ such that $\mathcal{B}$ full-matches $\mathcal{T}$.
  \end{itemize}
\label{fullmatch}
\end{lemma}
\begin{proof}
  By induction on $r$.  When $r=1$, the first condition holds immediately.  If the claim holds for $r$ and $c:NU(\mathcal{S})\rightarrow [1,r+1]$, we may apply Lemma \ref{fullmatchmaybe} and either reduce to IH or immediately give the second case.
\end{proof}

\begin{theorem}[$\mathbf{ACA^+}$]
    If $c:\mathcal{P}_{fin}(\mathbb{N})\rightarrow[1,r]$ is given then there are an $i\in[1,r]$ and an IP set $\mathcal{S}$ such that $c(S)=i$ for every $S\in\mathcal{S}$.
\end{theorem}
\begin{proof}
  The method is the same as Lemma \ref{fullmatchmaybe}.  Construct sequences $\mathcal{B}_2,\ldots,\mathcal{B}_n,\ldots$, $\mathcal{T}_1,\ldots,\mathcal{T}_n,\ldots$, and colorings $c_1,\ldots,c_n,\ldots$ by setting $c_1:=c$ and $\mathcal{T}_1:=\mathcal{S}$.  Given $c_i,\mathcal{T}_i$, apply Lemma \ref{fullmatch}; in the first case, we are done.  In the second, let $\mathcal{B}_{i+1},\mathcal{T}_{i+1}$ be the given witness and define $c_{i+1}$ on $NU(\mathcal{T}_{i+1})$ by setting $c_{i+1}(S):=\langle B,c_i(S)\rangle$ where $B\in\mathcal{B}_{i+1}$ is such that $c_i(S)=c_i(B)=c_i(S\cup B)$.

   Then for any $n$, we may find a sequence $\{B_i\}_{i\leq n}$ with $B_i\in\mathcal{B}_i$ and $c$ constant on $NU(\{B_i\}_{i\leq n})$.  By Weak K\"onig's Lemma, we may find an infinite sequence $\{B_i\}$ so that $c$ is constant on $NU(\{B_i\})$, as promised.
\end{proof}

\section{Difficult Examples}
In \citend{blass87}, a lower bound for the reverse mathematical strength of Hindman's Theorem is established by exhibiting a computable coloring of $\mathcal{P}_{fin}(\mathbb{N})$ which has no computable monochromatic IP set.  Specifically, two such colorings are given, one where every monochromatic IP set computes $\mathbf{0}'$ and one where no monochromatic IP set is computable in $\mathbf{0}'$.

In this section, we present computable colorings of $\mathcal{P}_{fin}(\mathbb{N})$ with various more specific properties.  We hope to serve three purposes: First, we will improve the recursion theoretic lower bound on Hindman's Theorem by giving a computable coloring of $\mathcal{P}_{fin}(\mathbb{N})$ with no $\Sigma_2$ monochromatic IP set.  Second, we will demonstrate that various stages in the proof of the previous section are optimal; if one hopes to give a proof of Hindman's Theorem within $\mathbf{ACA_0}$, this will help indicate where improvements are possible.  Finally, since these are the first new examples of colorings which are computationally difficult for Hindman's Theorem, we hope the relatively flexible nature of our method will spur the development of further progress.

We adopt a few notational conventions.  Whenever we write the union of two finite sets, say $B\cup C$, we always assume that $\max B<\min C$.  We say $\mathcal{S}$ \emph{generates} an IP set if $\mathcal{S}$ contains infinitely many pairwise disjoint elements.  (That is, $\mathcal{S}$ generates an IP set iff $NU(\mathcal{S})$ is an IP set.)  When we speak of one set $B$ containing a set $C$, we mean that $B=A_0\cup C\cup A_1$ with $\max A_0<\min C$, $\max C<\min A_1$ (and possibly $A_0,A_1$ or both empty).  Similarly, when we speak of an initial segment of $B$, we mean that $B=C\cup A_1$ with $\max C<\min A_1$.  We fix some ordering $\prec$ of $\B$ with order type $\omega$ so that if $\min B<\min C$ then $B\prec C$.

We will let $\mathcal{W}_1,\ldots,\mathcal{W}_i,\ldots$ be an enumeration of the computably enumerable subsets of $\B$, and for each $i,s$, define $\mathcal{W}_{i,s}$ to be a finite subset of $\B$  computable from $i,s$ such that $s\leq t$ implies $\mathcal{W}_{i,s}\subseteq \mathcal{W}_{i,t}$ and $\mathcal{W}_i=\bigcup_s \mathcal{W}_{i,s}$.

Before giving examples, we briefly describe our method, which is modeled on the finite injury priority argument.  (This idea was suggested to us by Carl Mummert.)  We will fix a list of conditions, indexed by the natural numbers, which we wish our coloring to satisfy; for instance, we might want to ensure that each of the countably many computably enumerable sets either fails to generate an IP set or generates one which is not monochromatic.

In this case, the $i$-th condition wishes to choose two elements of $\mathcal{W}_i$ and color them distinct colors.  However, since $\mathcal{W}_i$ is only computably enumerable, and we want our coloring to be computable, we must decide how to color a given set without being able to wait to see whether it will be in $\mathcal{W}_i$.  Instead, we will wait until some $W\in\mathcal{W}_{i,s}$ for some big enough $s$, and then color sets of the form $W\cup B$ where $\max B\geq s$.  If $\mathcal{W}_i$ generates an IP set, we are guaranteed that we can find a $B\in\mathcal{W}_i$ with $\max B\geq s$ (and $\max W<\min B$), and we will therefore have $W\cup B\in NU(\mathcal{W}_i)$ be an element colored according to our desired rule.

\subsection{A Computable Coloring with No Computably Enumerable Monochromatic IP Set}
To illustrate our method, we give a coloring with no computably enumerable monochromatic IP set.  Our method is similar to (though gives a weaker result than) Theorem 2.1 of \citend{blass87}.

\begin{theorem}
There is a computable $c:\B\rightarrow\{0,1\}$ such that if $\mathcal{S}$ is computably enumerable and generates an IP set then $NU(S)$ is not monochromatic.
\end{theorem}
\begin{proof}
For any $s$ and each $i\leq s$, we define $W^s_i$ to be least (with respect to $\prec$) such that:
\begin{itemize}
\item $W^s_{i}\in \mathcal{W}_{\lfloor i/2\rfloor,s}$
\item If $j<i$ and $W^s_j$ is defined then $\max W^s_j< \min W^s_i$
\end{itemize}
If there is no such element then $W^s_i$ is undefined.  Note that $\mathcal{W}_{\lfloor i/2\rfloor, s}$ is a finite set, so it is computable from $i,s$ whether $W^s_i$ exists, and if so, what the value of $W^s_i$ is.

Given $B\in \B$ with $\max B=s$, note that there are finitely many $W^s_i$ with $i\leq s$.  By checking each in turn, it is computable whether there is any $i$ such that $W^s_i$ is an initial segment of $B$.  From the definition of the $W^s_i$, there is at most one such $i$.  If there is no such $i$, set $c(B)=0$.  If there is such an $i$, set $c(B)=i\mod 2$.

\begin{claim}
  For each $i$, there is some $s$ such that $W^s_i=W^t_i$ for all $t\geq s$ (where both sides are undefined if either is).
\end{claim}
\begin{claimproof}
  By strong induction on $i$.  Let $s_0$ be large enough such that for all $j<i$, if $t\geq s_0$ then $W^{s_0}_j=W^t_j$.  If $\mathcal{W}_{\lfloor i/2\rfloor}$ contains any $W$ such that $ \min W> \max W^{s_0}_j$ for all $j<i$, there is a least such $W$.  There must be some $s$ such that $W\in W_{\lfloor i/2\rfloor,s}$, and it follows that $W^t_i=W$ for all $t\geq \max{s,s_0}$.  Otherwise, there is no such $W$, so $W^t_i$ is undefined for all $t\geq s_0$.  (When $i=0$ there are no $j<i$, so we may take $s_0=0$ and $W$ to be the least element of $W_0$ if $W_0$ is non-empty.)
\end{claimproof}

Then the following follows immediately:
\begin{claim}
  If $\mathcal{W}_e$ generates an IP set then there is some $s$ such that for all $t\geq s$, $W^t_{2e},W^t_{2e+1}$ are defined.
\end{claim}

Suppose $\mathcal{W}_e$ generates an IP set.  Then in particular, it contains some $W_0,W_1$ with $\max W_0<\min W_1$ such that for some $s$, $W^t_{2e}=W_0$ and $W^t_{2e+1}=W_1$ for all $t\geq s$.  Since $\mathcal{W}_e$ contains infinitely many pairwise disjoint elements, it must contain some $B$ with $\max B\geq s$ and $\min B>\max W_1$.  It follows that $c(W_0\cup B)=0$ and $c(W_1\cup B)=1$.  Since $W_0\cup B,W_1\cup B\in FS(W_e)$, it follows that $W_e$ does not generate a monochromatic IP set.
\end{proof}

\subsection{Computable Colorings with No Computably Enumerable Sets Half-Matched by Small Sets}
Here we show that there is no bound on the size of the finite set $\mathcal{B}$ found in Lemma \ref{halfmatch}.

\begin{theorem}
  For any $k$, there is a computable $c:\B\rightarrow\{0,1\}$ such that for any set $\mathcal{A}$ with size $\leq k$ and any computably enumerable set $\mathcal{S}$ such that $\mathcal{S}$ generates an IP set, $\mathcal{A}$ does not half-match $\mathcal{S}$.
\end{theorem}
\begin{proof}
Fix a computable sequence $\{\mathcal{A}_i,j_i\}$ where each $\mathcal{A}_i$ is a set of size $\leq k$, and such that whenever $\mathcal{A}$ is a set of size $\leq k$ and $j$ is an integer, there is an $i$ with $\mathcal{A}_i=\mathcal{A}$ and $j_i=j$.  The purpose of $j_i$ is to represent the computably enumerable set $\mathcal{W}_{j_i}$ from the enumeration fixed above.  In particular, if $\mathcal{A}$ is a set of size $\leq k$ and $\mathcal{W}$ is computably enumerable, there is an $i$ with $\mathcal{A}_i=\mathcal{A}$ and $\mathcal{W}_{j_i}=\mathcal{W}$.

For each $s$ and each $i\leq s$ and $u\in[0, k]$, we inductively define $W_{i,s}^u$ to be least satisfying the following properties:
  \begin{itemize}
  \item $i<\min W^u_{i,s}$
  \item $\max Z<\min W^u_{i,s}$ for all $Z\in\mathcal{A}_i$
  \item If $j<i$ and $W^{u'}_{j,s}$ is defined then $\max W^{u'}_{j,s}<\min W^u_{i,s}$
  \item If $u'<u$ and $W^{u'}_{i,s}$ is defined then $\max W^{u'}_{i,s}<\min W^u_{i,s}$
  \item $W^u_{i,s}\in \mathcal{W}_{j_i,s}$
  \end{itemize}
If there is no such $W^u_{i,s}$ then $W^u_{i,s}$ is undefined.  Note that $W_{i,s}^u$ is computable from $i,s,u$, since $\mathcal{W}_{j_i,s}$ is computable from $i,s$ (and in particular, the set of $i,s,u$ such that $W^u_{i,s}$ is defined is computable).

  A \emph{decomposition} of $B$ with $\max B=s$ is a tuple $i,u,Z,D$ such that $B=Z\cup W^u_{i,s}\cup D$ and neither $Z$ nor $D$ contains $W^{u'}_{i,s}$ for any $u'\neq u$.  
We often write that $Z\cup W^u_{i,s}\cup D$ is a decomposition of $B$ to mean that the tuple $i,u,Z,D$ is.  A decomposition is \emph{correct} if $Z\in\mathcal{A}_i$.  (Recall that when we write $Z\cup W^u_{i,s}\cup D$, we implicitly assume that $\max Z<\min W^u_{i,s}$ and $\max W^u_{i,s}<\min D$.)  Note that correctness of a decomposition is computable, since $\mathcal{A}_i$ is finite and computable from $i$, and $W^u_{i,s}$ is computable from $i,u,s$.

Observe that, for each $n$, there is a stage $s_n$ by which $W^u_{n,s_n}$ has stabilized for each $u\leq k$, in the sense that for all $t\geq s_n$, $W^u_{n,t}=W^u_{n,s_n}$ (where if one side is undefined then the other is as well).  When $W^u_{n,s_n}$ is defined, we call it $W^u_n$.
\begin{claim}
Let $c$ be a coloring, and suppose that for all $n$ and all $D$ with $\min D\geq s_n$, there is a $u\leq k$ such that either $W^u_n$ is undefined, or for all $Z\in\mathcal{A}_n$, $c(Z\cup W^u_n\cup D)\neq c(W^u_n\cup D)$.  Then $c$ satisfies the theorem.
\label{lastclaim}
\end{claim}
\begin{claimproof}
Let $\mathcal{A}$ be given with $|\mathcal{A}|= k$ and let $\mathcal{S}$ be computably enumerable and generate an IP set.  Choose $n$ such that $\mathcal{A}_n=\mathcal{A}$ and $\mathcal{W}_{j_n}=\mathcal{S}$.  Since $\mathcal{S}$ generates an IP set, $W^u_n$ is defined for all $u\leq k$, and we may find a $D\in\mathcal{S}$ with $\min D\geq s_n$.  Then for some $u$, $c(Z\cup W^u_n\cup D)\neq c(W^u_n\cup D)$ for all $Z\in\mathcal{A}$.  Therefore $\mathcal{A}$ does not half-match $W^u_n\cup D$, and since $W^u_n\cup D\in NU(\mathcal{S})$, it follows that $\mathcal{A}$ does not half-match $\mathcal{S}$. 
\end{claimproof}

We will construct $c$ so that it satisfies the preceeding claim.  A na\"ive attempt would be to simply decree that $c(Z\cup W^u_n\cup D)\neq c(W^u_n\cup D)$ for \emph{all} correct decompositions $Z\cup W^u_n\cup D$.  It's not hard to see, however, that this is too general.  If $B=Z\cup W^u_n\cup D=Z'\cup W^{u'}_{n'}\cup D'$ and both decompositions are correct then it might be that $c(W^u_n\cup D)\neq c(W^{u'}_{n'}\cup D')$, in which case we cannot color $B$ so that $c(B)\neq c(W^u_n\cup D)$ and also $c(B)\neq c(W^{u'}_{n'}\cup D')$.  Let us say, temporarily, that $Z,u$ \emph{conflicts} with $n'$ (over $n,D$) if there are $Z',u'$ so that $Z'\cup W^{u'}_{n'}\cup D'$ is a correct decomposition of $Z\cup W^u_n\cup D$; note that $u'$ is uniquely fixed by $D$.

When we have conflicting decompositions, we must have $n\neq n'$, by the definition of a decomposition.  If $n<n'$ then we must have $D'$ a proper final subset of $D$.  We illustrate this situation in Figure 1.  Note that this conflict only occurs when $Z'$ contains $W^u_n$ for exactly one $u$.  In particular, if we pick a fixed $D$ and $n<n'$, there are at most $|\mathcal{A}_{n'}|$ possible pairs $Z, u$ with $Z\in\mathcal{A}_n$ such that $Z,u$ conflicts with $n'$.  Since there are $k+1>k\geq |\mathcal{A}_{n'}|$ possible choices for $u$, this means there is some $u$ such for every $Z\in\mathcal{A}_n$, $Z,u$ does not conflict with $n'$ over $n,D$.

\begin{figure}[h]\centering  
  \begin{xy}
<25pt,-15pt>*\hbox{$Z$}; <65pt,-15pt>*\hbox{$W^u_{n,s}$};
    <190pt,-15pt>*\hbox{$D$}; <50pt,-45pt>*\hbox{$Z'$};
    <75pt,-45pt>*\hbox{$Z_0$}; <140pt,-45pt>*\hbox{$W^{u'}_{n',s}$};
    <240pt,-45pt>*\hbox{$D'$}; 
    <25pt,-45pt>*\hbox{$Z$};
\POS <0pt,0pt>\ar@{-}<300pt,0pt>
    \ar@{-}<0pt,-60pt> \POS <300pt,0pt>\ar@{-}<300pt,-60pt> \POS
    <0pt,-30pt> \ar@{-}<300pt,-30pt> \POS
    <50pt,0pt>\ar@{-}<50pt,-30pt> \POS <80pt,0pt>\ar@{-}<80pt,-30pt>
    \POS<0pt,-60pt>\ar@{-}<300pt,-60pt>
    \POS<100pt,-30pt>\ar@{-}<100pt,-60pt>
    \POS<180pt,-30pt>\ar@{-}<180pt,-60pt>
    \POS<50pt,-30pt>\ar@{--}<50pt,-60pt>
  \end{xy}
  \caption{Two decompositions of the same set $B$}
\end{figure}

There is a remaining obstacle, namely that, for various values of $Z$ and $u$, the pair $Z,u$ could conflict with multiple values of $n'$.  Our solution is to use a stronger notion, \emph{blocking}, and arrange (see Claim \ref{blockingclaim}) that we need only worry about the largest $n'$ which is a source of conflicts.

We now make this precise.  Consider triples $i,u,D$ (viewed as referring to the set $W^u_{i,\max D}\cup D$); we define the \emph{blocked} triples $i,u,D$ by induction on the length of $D$.  The triple $i,u,D$ is \emph{blocked} by $i'$ if there exist $Z',i',u',D'$ such that:
\begin{itemize}
  \item $\max D=\max D'$,
  \item $W^{u'}_{i',\max D}\cup D'$ is a final segment of $D$,
  \item $W^u_{i,\max D}\cup D$ is a final segment of $Z\cup W^{u'}_{i',\max D}\cup D'$,
  \item The triple $i',u',D'$ is unblocked,
  \item $Z'\cup W^{u'}_{i',\max D}\cup D'$ is a correct decomposition, and
  \item If $Z'$ contains $W^{u^*}_{i,\max D}$ then $u^*=u$.
\end{itemize}
  Note that when this occurs, $i<i'$.  When $B=Z\cup W^u_{i,s}\cup D$ is a correct decomposition, we say it is blocked by $i'$ iff $i,u,D$ is blocked by $i'$.

  \begin{claim}
    For any $B$, there is at most one correct unblocked decomposition.
  \end{claim}
  \begin{claimproof}
  \label{intervene}
    Suppose $B=Z\cup W^u_{i,s}\cup D=Z'\cup W^{u'}_{i',s}\cup D'$ give two correct decompositions, with $Z$ a proper intial segment of $Z'$.  If $W^{u'}_{i',s}\cup D'$ is not blocked then by definition, $i,u,D$ is blocked by $i'$.
  \end{claimproof}

Since correctness is computable, we may identify the unblocked decompositions of $B$ by examining all possible decompositions of all sets $B'$ with $\max B'=\max B$.  There are finitely many such sets $B'$, and therefore finitely many such decompositions.  In particular, given $B$, we may computably determine whether there is a correct unblocked decomposition, and if so, what it is.
   
  We now define our coloring inductively.  Let $B$ be given, and suppose $c(B)$ has been decided for all proper final segments of $B$.  Let $Z\cup W^u_{i,s}\cup D$ be the correct, unblocked decomposition, if there is one.  Then set $c(Z\cup W^u_{i,s}\cup D)=1-c(W^u_{i,s}\cup D)$.  If there is no correct unblocked decomposition, set $c(B)=0$.


\begin{claim}
  Suppose $n,v,B$ is blocked by $i$ while $n,v',B$ is blocked by $i'$.  Then $i=i'$.
  \label{blockingclaim}
\end{claim}
\begin{claimproof}
  Suppose $i\neq i'$; without loss of generality, we may assume $i<i'$.  Let $s=\max B$.  There exist $Z,u,D$ and $Z',u',D'$ witnessing the blocking.  We will show that $Z',i',u',D'$ witnesses the blocking of $i,u,D$.  
  
  We certainly have $\max D=s=\max D'$.  Since $W^u_{i,s}\cup D$ and $W^{u'}_{i',s}\cup D'$ are both proper final segments of $B$ with $\max W^u_{i,s}<\min W^{u'}_{i',s}$, it follows that $W^{u'}_{i',s}\cup D'$ is a proper final segment of $D$.  Since $W^u_{i,s}\cup D$ is a proper final segment of $W^v_{n,s}\cup B$, which is in turn a proper final segment of $Z'\cup W^{u'}_{i',s}\cup D'$, we have that $W^u_{i,s}\cup D$ is a proper final segment of $Z'\cup W^{u'}_{i',s}\cup D'$.  Since $Z',i',u',D'$ blocks $n,v,B$, it must be that $i',u',D'$ is unblocked.  By assumption, $Z'\cup W^{u'}_{i',s}\cup D'$ is a correct decomposition.
  
  Finally, suppose $Z'$ contains $W^{u^*}_{i,s}$ for some $u^*$; since $W^v_n\cup B$ is a proper final segment of $Z'\cup W^{u'}_{i',s}\cup D'$ and $\max W^v_n<\min W^{u^*}_{i,s}$, it must be that $W^{u^*}_{i,s}$ is contained in $B$.  Since $W^v_n\cup B$ is a proper final segment of $Z\cup W^u_{i,s}\cup D$, it must be that $W^{u^*}_{i,s}$ is contained in $Z\cup W^u_{i,s}\cup D$.  Since $Z\cup W^u_{i,s}\cup D$ is a decomposition, $u^*=u$.
  
  These conditions show that $i,u,D$ is blocked, contradicting the assumption.  So we must have $i= i'$.
  \end{claimproof}
  
  So, holding $B$ fixed, there is at most one $i$ such that there exist $A,v$ so that $A\cup W^v_{n,\max B}\cup B$ is blocked by $i$.  In order for $A\cup W^v_{n,\max B}\cup B$ to be blocked by $i$, there must be a $Z\in\mathcal{A}_i$ such that $W^v_n$ is contained in $Z$, and for $v\neq v'$, $W^{v'}_n$ is not contained in $Z$.  Since $|\mathcal{A}_i|=k$, there are at most $k$ values of $v$ for which any $A\cup W^v_{n,\max B}\cup B$ is blocked.  Therefore for some $v\leq k$, $A\cup W^v_{n,\max B}\cup B$ is a correct unblocked decomposition for all $A\in\mathcal{A}_n$, and therefore $c(A\cup W^v_{n,\max B}\cup B)= 1-c(W^v_n\cup B)$.  We may now apply Claim \ref{lastclaim}.
\end{proof}

\subsection{A Computable Coloring with No Computably Enumerable Full-Matched Sets}

Here we show that the first clause in Lemma \ref{fullmatchmaybe} is necessary by presenting a computable coloring in which there is no finite set $\mathcal{B}$ and computable, or even computably enumerable, IP set $\mathcal{T}$ such that $\mathcal{B}$ full-matches $\mathcal{T}$.

\begin{theorem}
  There is a computable $c:\B\rightarrow\{0,1\}$ such that for any finite set $\mathcal{B}$ and any computably enumerable set $\mathcal{S}$ such that $\mathcal{S}$ generates an IP set, $\mathcal{B}$ does not full-match $\mathcal{S}$.
\end{theorem}
\begin{proof}
For each $s$ and each $i\leq s$ and $u\in\{0,1\}$, we inductively define $W_{i,s}^u$ to be least satisfying the following properties:
  \begin{itemize}
  \item $i<\min W^u_{i,s}$
  \item If $j<i$ and $W^{u'}_{j,s}$ is defined then $\max W^{u'}_{j,s}<\min W^u_{i,s}$
  \item If $W^0_{i,s}$ is defined then $\max W^0_{i,s}<\min W^1_{i,s}$
  \item $W^u_{i,s}\in \mathcal{W}_{i,s}$
  \end{itemize}
If there is no such $W^u_{i,s}$ then $W^u_{i,s}$ is undefined.  Since $\mathcal{W}_{i,s}$ is a finite set computable from $i,s$, $W^u_{i,s}$ is computable from $u,i,s$.

A \emph{primary $s$-decomposition} of $B$, where $s=\max B$, is a tuple $i,u,Z,D$ such that $B=Z\cup W^u_{i,s}\cup D$, neither $Z$ nor $D$ contains $W^{1-u}_{i,s}$ as a subsequence, and there is no primary $s$-decomposition of $D$.  Clearly there is at most one primary $s$-decomposition of $B$.  Note that since there are only finitely many decompositions of $B$, we need search only finitely many possibilities to identify whether there is a primary $s$-decomposition of $B$, and if so, what it is.

We say $B$ \emph{contains} $i$ with polarity $v$ if there is a primary $\max B$-decomposition $j,u,Z,D$ of $B$ with either $i=j$ and $v=u$, or $i$ contained in $Z$ with polarity $|v-u|$.  Observe that whenever $B$ contains $i$, $B=Z\cup W^u_{i,t}\cup D$ for some $t\leq \max B$.

We now define our coloring inductively.  Let $B$ be given, and suppose we have already decided $c(B')$ whenever $B'$ is a proper initial segment of $B$.  If $B$ has a primary $s$-decomposition $B=Z\cup W^u_{i,s}\cup D$, we set $c(B)=c(Z)$ if $u=0$ and $c(B)\neq c(Z)$ if $u=1$.  If there is no primary $s$-decomposition of $B$, we set $c(B)=0$.

\begin{claim}
  For each $i$, there is some $s$ such that $W^s_i=W^t_i$ for all $t\geq s$ (where both sides are undefined if either is).
\end{claim}

Let $\mathcal{B}$ be a finite set such that for all $A\in \mathcal{B}$, $\max A\leq i$ and let $s,W^0_i, W^1_i$ be such that for all $t\geq s$, $W^u_{i,s}=W^u_i$.  It is easy to see that for any $B$ with $\min B\geq s$ there is a $v_B$ such that, $A\cup W^u_i\cup B$ contains $i$ with polarity $|v_B-u|$ for all $A\in \mathcal{B}$.  

\begin{claim}
  For all $B$ with $\min B\geq s$, $c(A\cup W^{v_B}_i\cup B)=c(A)$ and $c(A\cup W^{1-v_B}_i\cup B)\neq c(A)$.
\end{claim}
\begin{claimproof}
  By induction on the length of $B$.  Let $D=A\cup
  W^u_i\cup B$.  $A\cup W^u_i\cup B$ gives a primary $\max
  B$-decomposition of $D$ unless $B$ has a primary $\max
  B$-decomposition, so $D$ must have a primary $\max B$-decomposition
  $Z\cup W^{u'}_j \cup B'$.  If we just have $j=i$, the claim follows
  immediately from the definition of the coloring.

  Otherwise, if $u'=0$ then $c(D)=c(Z)$ and $Z$ contains $i$ with
  polarity $|v_B-u|$; by IH applied to $Z\setminus A\cup W^u_i$,
  $c(D)=c(Z)=c(A)$ if $u=v_B$ and $c(D)=c(Z)\neq c(A)$ if $u\neq v_B$.  If
  $u'=1$ then $c(D)\neq c(Z)$ and $Z$ contains $i$ with polarity
  $1-|v_B-u|$; by IH applied to $Z\setminus A\cup W^u_i$, $c(D)\neq
  c(Z)\neq c(A)$ if $u=v_B$, so $c(D)=c(A)$, and $c(D)\neq c(Z)=c(A)$ if
  $u\neq v_B$.
\end{claimproof}

So suppose $\mathcal{A}$ full-matched $NU(\mathcal{W})$ with $\mathcal{W}$ computably enumerable.  Then for some $i$ such that $\max A\leq i$ for all $A\in\mathcal{A}$, we have $\mathcal{W}=\mathcal{W}_i$.  If $\mathcal{W}_i$ generated an IP set, there would be a $B\in \mathcal{W}_i$ with $\min B\geq s$, and $W^0_i,W^1_i\in\mathcal{W}_i$ such that either $A$ failed to full-match $W^0_i\cup B$ or $A$ failed to full-match $W^1_i\cup B$.  In either case, since both $W^0_i\cup B$ and $W^1_i\cup B$ belong to $NU(\mathcal{W}_i)$, $\mathcal{A}$ fails to full-match $NU(\mathcal{W}_i)$.
\end{proof}

\subsection{A Computable Coloring with No $\Sigma_2$ Monochromatic IP Set}
\begin{theorem}
  There is a computable $c:\B\rightarrow\{0,1\}$ such that if $\mathcal{S}$ is a $\Sigma_2$ set generating an IP set then $NU(\mathcal{S})$ is not monochromatic.
\label{hardsigma2}
\end{theorem}
\begin{proof}
Fix an enumeration of all $\Sigma_2$ formulas
\[\phi_i(Z)=\exists x\forall y R_i(x,y,Z).\]
We will sometimes conflate $\phi_i$ with $\{Z\in\B\mid\phi_i(Z)\}$ (for example, by writing $NU(\phi_i)$).

We arrange pairs $(i,n)$ with $n<i+1$ in lexicographic order (so $(j,m)<(i,n)$ iff $j<i$ or $j=i$ and $m<n$).  For each pair $(i,n)$, we define the $i,n$-candidates and $T_{i,n}$, the $i,n$-witness, simultaneously by induction.

We will now define the key building blocks of our argument, the \emph{candidates} and \emph{witnesses}.  The main point of an $i$-candidate is that it will satisfy $\phi_i$; a secondary point is that its smallest element is largest enough to give bounds on the existential quantifiers needed to justify all the earlier witnesses.  In other words, a candidate should ``see'' all the earlier witnesses.  A witness, in turn, is just the smallest candidate.  (We could dispense with the notion of a candidate, and discuss only witnesses; the notion of a candidate is used to simplify the proofs of some claims.)

\begin{definition}
  $T$ is an $i,n$-candidate if:
  \begin{itemize}
  \item $\phi_i(T)$
  \item For each $(j,m)<(i,n)$ such that the least $j,m$-candidate $T_{j,m}$ is defined, $\exists x\leq \min T\forall y R_j(x,y,T_{j,m})$
    \item For all $(j,m)<(i,n)$ such that the least $j,m$-candidate $T_{j,m}$ is defined, $\max T_{j,m}<\min T$
  \end{itemize}

  We define $T_{i,n}$, the $i,n$-witness, to be the least $i,n$-candidate if there is one, and undefined otherwise.
\end{definition}
Note that if $\phi_i$ generates an IP set then all the $i,n$-witnesses are defined.

We will also need certain approximations to the $i,n$-witnesses.
\begin{definition}
  Let integers $p,q$ be given.    $T$ is a $p,q,i,n$-candidate if:
  \begin{itemize}
    \item $\max T<p$
  \item $\exists x\leq p\forall y\leq q R_i(x,y,T)$
  \item For all $(j,m)<(i,n)$ such that the least $p,q,j,m$-candidate $T^{p,q}_{j,m}$ is defined, $T^{\min T,q}_{j,m}=T^{p,q}_{j,m}$
    \item For all $(j,m)<(i,n)$ such that the least $p,q,j,m$-candidate $T^{p,q}_{j,m}$ is defined, $\max T^{p,q}_{j,m}<\min T$
  \end{itemize}

  We define $T^{p,q}_{i,n}$, the $p,q,i,n$-witness, to be the least $p,q,i,n$-candidate if there is one, and undefined otherwise.
\end{definition}
Note that there are only finitely many sets with $\max T<p$, and therefore only finitely many possible candidates for $T^{p,q}_{i,n}$; in particular, the set of $p,q,i,n$ such that $T^{p,q}_{i,n}$ exists is computable, and $T^{p,q}_{i,n}$ can be computed from $p,q,i,n$.

\begin{claim}
  If $p\leq p'$ and $T^{p,q}_{j,m}=T^{p',q}_{j,m}$ for all $(j,m)<(i,n)$ then $T^{p',q}_{i,n}\preceq T^{p,q}_{i,n}$.
\end{claim}
\begin{claimproof}
  It suffices to show that $T^{p,q}_{i,n}$ is a $p',q,i,n$-candidate.  Certainly if $\exists x\leq p \forall y\leq q R_i(x,y,T^{p,q}_{i,n})$ then there is such an $x\leq p'$ as well.  The remaining conditions hold by assumption.
\end{claimproof}

\begin{claim}
  If $p\leq p'\leq p''$ and $T^{p,q}_{j,m}=T^{p'',q}_{j,m}$ for all $(j,m)<(i,n)$ then $T^{p',q}_{j,m}=T^{p,q}_{j,m}$ for all $(j,m)\leq (i,n)$.
\end{claim}
\begin{claimproof}
  Suppose not.  Let $(j,m)$ be least such that $T^{p,q}_{j,m}\neq T^{p',q}_{j,m}$.  Applying the preceeding lemma to $p,p'$ and to $p',p''$, we have $T^{p',q}_{j,m}\prec T^{p,q}_{j,m}=T^{p'',q}_{j,m}\prec T^{p',q}_{j,m}$, which is impossible.
\end{claimproof}

We define a coloring of $\B$ as follows.  Let $B\in \B$ be given with $\max B=s$; we may assume $c(B')$ is decided for all $B'$ with $\max B'<s$ and for all proper final segments of $B$.  We will attempt to color $B$ in a series of stages, indexed by $i\leq s$.  At stage $i$, we ask whether there exist $A,D$ such that:
\begin{itemize}
\item $A\cup D=B$,
\item $\max A<\min D$, and
\item $A=T^{\min D,\max D}_{i,n}$ for some $n<i+1$.
\end{itemize}
If all these conditions are met, we set $c(B)=1-c(D)$ for the longest such $D$, and say that $B$ is $i,A,D$-colored.  Otherwise, we do not color $B$ at stage $i$.  This is computable since there are only finitely many possible divisions $B=A\cup D$ which need to be checked, and checking if $A=T^{\min D,\max D}_{i,n}$ is computable.

If $B$ is not colored at any stage $i\leq s$, we arbitrarily set $c(B)=0$.

For each $i$, we wish to show that if $\phi_i$ generates an IP set then $c$ is not monochromatic on $FU(\phi_i)$.  So suppose $\phi_i$ generates an IP set.  Choose $p$ such that for each $T'\preceq T_{i,i}$, if $\exists x\forall y R_i(x,y,T')$ then $\exists x\leq p\forall y R_i(x,y,T')$.  Since $\phi_i$ generates an IP set, we may find an $A$ with $\min A\geq p$ and $\phi_i(A)$.  Now let $q$ be large enough that for each $j\leq i$, each $T'\preceq T_{i,i}$ such that $\neg\phi_j(T')$, and each $x\leq \min A$, there is a  $y\leq q$ such that $\neg R_j(x,y,T')$.  Again we may find $B$ such that $\phi_i(B)$ and $\max B\geq q$.  In particular, when $j\leq i$, $T'\preceq T_{i,i}$, $\exists x\forall y R_j(x,y,T')$ holds iff $\exists x\leq \min A\forall y\leq \max B R_j(x,y,T')$ holds, and therefore $T_{j,m}=T^{\min A,\max B}_{j,m}$ for all $(j,m)\leq (i,i)$.

We will show that for some $n<i+1$, $T_{i,n}\cup A\cup B$ is $i,T_{i,n},A\cup B$-colored.  This means $c(T_{i,n}\cup A\cup B)\neq c(A\cup B)$, and therefore $NU(\phi_i)$ is not monochromatic.  Since $T_{i,n}=T^{\min A,\max B}_{i,n}$, it suffices to show that for some $n<i+1$, $T_{i,n}\cup A\cup B$ is not $j,T',D$-colored for any $j<i$ with $T'\neq T_{i,n}$ or $i,T',D$-colored for any $T'$ a proper initial segment of $T'_{i,n}$.

\begin{claim}
  If $T'$ is a proper initial segment of $T_{i,n}$ and $j<i$ then $T_{i,n}\cup A\cup B$ is not $j,T',D$-colored, where $T'\cup D=T_{i,n}\cup A\cup B$.
\end{claim}
\begin{claimproof}
  Since $T_{j',m}^{\min T',\max D}=T_{j',m}^{\min A,\max D}=T_{j',m}$ for all $(j',m)<(i,n)$ and $\min T'\leq \min D\leq \min A$, it follows that $T_{j',m}^{\min D,\max D}=T_{j',m}$ for all $(j',m)<(i,n)$.  In particular, since $T'$ is a proper initial segment of $T_{i,n}$, we cannot have $T'=T_{j,m}$ for any $m$.  Therefore $T_{i,n}\cup A\cup B$ is not $j,T',D$-colored.
\end{claimproof}

\begin{claim}
  If $T'$ is a proper initial segment of $T_{i,n}$ and $j<i$ then $T_{i,n}\cup A\cup B$ is not $i,T',D$-colored, where $T'\cup D=T_{i,n}\cup A\cup B$.
\end{claim}
\begin{claimproof}
  If $\phi_i(T')$ then $T'$ would be an $i,n$-candidate with $T'\prec T_{i,n}$, contradicting leastness of $T_{i,n}$.  So $\neg\phi_i(T')$, and therefore $\forall x\leq \min A\exists y\leq \max B\neg R_j(x,y,T')$.  Since $\max B=\max D$ and $\min D\leq \min A$, also $\forall x\leq \min D\exists y\leq \max D\neg R_j(x,y,T')$, so $T'$ cannot be $T^{\min D,\max D}_{i,m}$ for any $m$.
\end{claimproof}

It is still possible for $T_{i,n}\cup A\cup B$ to be $j,T',D$-colored by some $j<i$ when $T'$ is a proper end-extension of $T_{i,n}$.  We will show that each $j$ does so for at most one $n<i+1$.  

\begin{claim}
  If $j<i$ and $T_{i,n}\cup A\cup B$ is $j,T',D$-colored where $T'$ is a proper end-extension of $T_{i,n}$ then $T'=T^{\min D,\max D}_{j,m}$ is least such that $T_{j,m}$ is undefined.
\end{claim}
\begin{claimproof}
  By definition, $T'=T^{\min D,\max D}_{j,m}$ for some $m<j+1$.  If $T_{j,m'}$ is defined for some $m'<j+1$ then, since $\min A\leq \min D$, $T_{j,m'}=T^{\min D,\max D}_{j,m'}\neq T'$.  If $m'<m$ is such that $T_{j,m'}$ is undefined, in order for $T'\cup D$ to be $j,T',D$-colored, we would have to have $T_{j,m'}^{\min T',\max D}=T_{j,m'}^{\min D,\max D}$.  But $\min T'=\min T_{i,n}$ and $\max D=\max B$, so $T_{j,m'}^{\min T',\max D}=T_{j,m'}^{\min T_{i,n},\max B}$ is undefined.  Therefore $m$ is least such that $T_{j,m'}$ is undefined.
\end{claimproof}

So suppose there are distinct $n,n'<i+1$ such that $T_{i,n}\cup A\cup B$ is $j,T',D'$-colored while $T_{i,n'}\cup A\cup B$ is $j,T'', D''$-colored.  Without loss of generality, assume $T'\prec T''$.  Then $\max D'=\max D''=\max B$ and $\min A\leq \min D'\leq \min D''$.  Let $m$ be least such that $T_{j,m}$ is undefined.  Then $T'$ is a $\min D'',\max D'',j,m$-candidate.  Since $T'\prec T''$, it follows that $T''$ cannot be $T^{\min D'',\max D''}_{j,m}$.

Therefore for each $j$,  there is at most one $n$ such that $T_{i,n}\cup A\cup B$ is $j,T',D$-colored.  This means there are at most $i$ choices of $n$ such that $T_{i,n}\cup A\cup B$ is $j,T',D$-colored for any $j<i$, and since there are $i+1$ possible values for $n$, there is some $n$ such that $T_{i,n}\cup A\cup B$ is not $j,T',D$-colored for any $j<i$, and therefore $T_{i,n}\cup A\cup B$ is $i,T_{i,n},A\cup B$-colored, as desired.
\end{proof}

\section{Conclusion}
The results of the previous section still leave a significant gap in the strength of Hindman's Theorem; in particular, while we do not see how to prove Hindman's Theorem in $\mathbf{ACA_0}$, we cannot rule out the possibility that there is such a proof.

\bibliographystyle{jflnat}
\bibliography{../../Bibliographies/main}

\end{document}